\documentclass[11pt]{article}
\usepackage[utf8]{inputenc}
\usepackage[english]{babel}
\usepackage{hyperref}  
\usepackage{enumerate}
\usepackage[colorinlistoftodos,prependcaption,textsize=tiny]{todonotes}
\usepackage[numbib]{tocbibind}
\usepackage{graphicx}
\usepackage{anysize}

\usepackage{amsmath,amssymb,amsthm}
\newtheorem*{rep@theorem}{\rep@title}
\newcommand{\newreptheorem}[2]{%
\newenvironment{rep#1}[1]{%
 \def\rep@title{#2 \ref{##1}}%
 \begin{rep@theorem}}%
 {\end{rep@theorem}}}

\newreptheorem{theorem}{Theorem}
\newtheorem{theorem}{Theorem}[section]

\newtheorem{corollary}[theorem]{Corollary}
\newtheorem{lemma}[theorem]{Lemma}
\newtheorem{question}[theorem]{Question}

\newtheorem{remark}[theorem]{Remark}

\newtheorem{conjecture}[theorem]{Conjecture}
\newtheorem{exam}[theorem]{Example}

\usepackage{parskip}

\usepackage{parskip}
\usepackage{setspace}

\def \Z{{\mathbb Z}}
\def \Zbb{{\mathbb Z}}
\def \Cbb{{\mathbb C}}

\def\bZ{\mathbb{Z}}

\def\bfo{\mathbf{0}}
\def\bfm{\mathbf{m}}

\def\bfx{\mathbf{x}}
\def\bfy{\mathbf{y}}
\def\bfz{\mathbf{z}}

\def\bfu{\mathbf{u}}

\def\bfv{\mathbf{v}}

\def\mod{\textsf{mod }}
\def\orb{\textsf{orb}}

\def\l({\left(}
\def\r){\right)}

\def\lv{\left\vert}
\def\rv{\right\vert}

\def\lo{\left\{}
\def\ro{\right\}}

\def\lV{\left\Vert}
\def\rV{\right\Vert}

\begin{document}
\title{On a theorem of Mattila in the p-adic setting}
\author{Boqing Xue\thanks{Institute of Mathematical Sciences, ShanghaiTech University. Email: xuebq@shagnhaitech.edu.cn}\and Thang Pham\thanks{University of Science, Vietnam National University, Hanoi. Email: thangpham.math@vnu.edu.vn}\and 
Le Q. Hung\thanks{Hanoi University of Science and Technology, Hanoi. Email: hung.lequang@hust.edu.vn}\and Le Q. Ham\thanks{Vietnam Institute of Educational Sciences. Email: hamlq2022@gmail.com}\and Nguyen D. Phuong\thanks{People’s Security Academy, Ha
noi, Vietnam. Email: duyphuong78@gmail.com}  \thanks{Author ordering is randomized.}}
\maketitle
\begin{abstract}
Let $A, B$ be subsets of $(\mathbb{Z}/p^r\mathbb{Z})^2$. In this note, we provide conditions on the densities of $A$ and $B$ such that $|gA-B|\gg p^{2r}$ for a positive proportion of $g\in SO_2(\mathbb{Z}/p^r\mathbb{Z})$. The conditions are sharp up to constant factors in the unbalanced case, and the proof makes use of tools from discrete Fourier analysis and results in restriction/extension theory.
\end{abstract}
\section{Introduction}
We begin with the following question. 
\begin{question}
Let $p$ be a prime, and $A$, $B$ be subsets in $\mathbb{F}_p^d$. Under what conditions on $A$ and $B$ do we have $|gA-B|\gg p^d$
for a positive proportion of $g\in SO_d(\mathbb{F}_p)$? 
\end{question}

The original version of this question was posed in the Euclidean setting and has a well-known history, which we will not detail here. It is necessary to mention that for $A, B\subset \mathbb{R}^d$ with $\dim_H(A)+\dim_H(B)>d$, Mattila \cite{Mattila} proved that if $$\frac{(d-1)}{d}\dim_H(A)+\dim_H(B)>d ~~\mbox{or}~~\dim_H(B)>\frac{d+1}{2},$$ then $\mathcal{L}^d(gA-B)>0$ for almost all orthogonal matrices $g$ in $\mathbb{R}^d$. Here, by $\dim_H$ and $\mathcal{L}(\cdot )$ we mean the Hausdorff dimension and the Lebesgue measure in $\mathbb{R}^d$, respectively. In two dimensions with $\dim_H(A)=\dim_H(B)=t$, he asked in the survey paper \cite{M2023} whether or not the condition $t>5/4$ would be enough to guarantee that $\mathcal{L}(gA-B)>0$ by using the techniques due to Guth, Iosevich, Ou, Wang in \cite{alex-fal} on the Falconer distance problem. 

In a recent paper, Pham and Yoo \cite{P.Y.2023} answered this question affirmatively in the prime field setting by showing that if $|A|=|B|\gg p^{5/4}$ and $p\equiv 3\pmod 4$, then, for a positive proportion of $g\in SO_2(\mathbb{F}_p)$, we have $|gA-B|\gg p^2$. 
When $p\equiv 1\pmod 4$, their method implies a weaker exponent, namely, $3/2$ instead of $5/4$. Notice that these exponents cannot be reduced to a number less than $1$, to see this, one takes two sets $A$ and $B$ being on a circle centered at the origin of radius $1$, then it is clear that $|gA-B|\sim |A||B|$. 

In this note, we are interested in studying this topic in the plane over a finite p-adic ring. In particular, we consider the following question. Write $\delta_A=|A|/p^{2r}$ for any set $A\subset (\Z/p^r\Z)^2$. 

\begin{question}
Let $A, B$ be sets in $(\mathbb{Z}/p^r\mathbb{Z})^2$. 
Under what conditions on $\delta_A$ and $\delta_B$ do we have that $|gA-B|\gg p^{2r}$ for a positive proportion of $g\in SO_2(\mathbb{Z}/p^r\mathbb{Z})$? 
\end{question}

The following examples show that in order to have $|gA-B|\gg p^{2r}$, the sets $A$ and $B$ cannot both be small.



\begin{exam}\label{ex1.3}
    Let $X=\{\bfx \in (\Z/p^r\Z)^2:\, \bfx\equiv \bfo \pmod p\}$.
    Let $A$ (and $B$) be disjoint unions of $m$ (and $n$, respectively) cosets of $X$, where $1\leq m,n\leq p^2$. So, $|A|=mp^{2r-2}$ and $|B|=np^{2r-2}$. One sees that $|g A|\leq mp^{2r-2}$, and $|gA-B|\leq mnp^{2r-2}$.
\end{exam}

Our main theorem reads as follows. 

\begin{theorem}\label{thm:9.9}
    Let $p$ be a prime with $p\equiv 3 \pmod 4$, and $r$ be a positive integer. Let $A,B \subset (\Zbb /p^r \Zbb)^2$ be such that $\delta_A^{1/2}\cdot \delta_B\ge 2p^{-1}$. Then for a positive proportion of $g \in SO_2(\Zbb /p^r \Zbb)$, we have
    \[ \lv gA -B\rv \gg p^{2r}. \] 
\end{theorem}

Based on Example \ref{ex1.3}, the condition $\delta_A^{1/2}\cdot\delta_B\ge 2p^{-1}$ is sharp up to constant factors. Indeed, if we take $m=1$ and $n=\gamma (p)p^2$ with $\gamma(p)$ tending to $0$ arbitrarily slowly as $p$ tends to infinity, then $\delta_A^{1/2}\delta_B=\gamma(p)p^{-1}$ and $|gA-B|\leq \gamma(p)p^{2r}$. Example \ref{ex1.3} is most interesting when $r>1$ since it shows the sharpness of the exponent $1/2$ on the density condition of $A$. 


To prove this result, we make use of tools from discrete Fourier analysis and results in restriction/extension theory. As in the finite field setting \cite{P.Y.2023}, the key estimate in our proof is the following sum
 \[\sum_{\bfm \ne \bfo}\sum_{
    \bfm' \in V_\bfm}\lv \widehat{A}(\bfm)\rv^2 \lv \widehat{B}(\bfm ')\rv^2,\]
    where
    $V_{\mathbf{m}}:=\{g \mathbf{m}\colon g\in SO_2(\mathbb{Z}/p^r\mathbb{Z})\}$. There are several approaches one can use to bound this sum:
    
\begin{enumerate}
\item Using the Plancherel formula is a standard argument and often yields a weak bound.
\item A more advanced approach employs restriction/extension estimates associated with circles. While such estimates are well understood in the finite field setting, see \cite{CEHIK10, KS}, for example, they have only recently been developed in the $p$-adic setting \cite{P.X}. This method will be used to prove Theorem \ref{thm:9.9}.
\item When $r=1$, a novel method, combining the $L^2$ norm of the distance function from \cite{mu2} with a double-counting argument, is introduced in \cite{P.Y.2023}.
Extending this approach to the $p$-adic setting, however, presents challenges due to the limited understanding of incidence structures, particularly those involving points and planes in $(\mathbb{Z}/p^r\mathbb{Z})^3$. We hope to address this issue in the near future.
\end{enumerate}

Although Theorem \ref{thm:9.9} is not as strong as Pham and Yoo's result in \cite{P.Y.2023}, it establishes that if $\delta_A^{1/2}\cdot \delta_B \ge 2p^{-1}$, then $|gA-B|\gg p^2$ for a positive proportion of $g\in SO_2(\mathbb{Z}/p\mathbb{Z})$. This aligns directly with Mattila’s theorem in \cite{Mattila} with $d=2$, as mentioned above.

Notice that when $p\equiv 1\pmod 4$, the restriction/extension estimates are weaker, so the proof of Theorem \ref{thm:9.9} implies the condition of $\delta_A^{1/2}\cdot \delta_B\ge 2p^{-1/2}$. This is worse than the next theorem, which will be proved by using the Plancherel formula directly.

\begin{theorem}\label{thm:9.10}
    Let $p$ be a prime with  $p\equiv 1 \pmod 4$, and $r$ be a positive integer. Let $A,B \subset (\Zbb /p^r \Zbb)^2$ be such that $\delta_A\cdot \delta_B \ge 2p^{-1}$. Then for a positive proportion of $g\in SO_2(\Zbb /p^r \Zbb)$, we have
    \[ \lv gA -B\rv \gg p^{2r}. \] 
\end{theorem}
In Example \ref{ex1.3}, by taking $A=B=X+Y$, where $Y\subset \{(x_1, x_2)\in (\mathbb{Z}/{p}\mathbb{Z})\times (\mathbb{Z}/{p}\mathbb{Z})\colon x_1^2+x_2^2\equiv 1\pmod p \}$, then, for any $g\in SO_2(\mathbb{Z}/p^r\mathbb{Z})$, we have $|A|=|B|=|Y|\cdot p^{2r-2}$ and $|gA-B|\ll |Y|^2\cdot p^{2r-2}$. Therefore, it is reasonable to make the following conjecture on the balanced case. 
\begin{conjecture}
    Let $p$ be an odd prime, and $r$ be a positive integer. Let $A,B \subset (\Zbb /p^r \Zbb)^2$ be such that $\delta_A=\delta_B\gg p^{-1}$. Then for a positive proportion of $g \in SO_2(\Zbb /p^r \Zbb)$, we have
    \[ \lv gA -B\rv \gg p^{2r}. \] 
\end{conjecture}
This conjecture appears highly challenging and could be as difficult as the Erd\H{o}s-Falconer distance problem. Regarding further open questions, a natural direction is to explore similar problems in other group settings or in higher dimensions. To generalize Theorem \ref{thm:9.9} to higher dimensions, the first step is to extend the results in Section 2, but we found these extensions too complicated in this setting. This suggests exploring alternative approaches to address the problem.

{\bf Notations:} In this paper, by $X\gg Y$ we mean $X\ge CY$ for some absolute positive constant $C$, and $X\sim Y$ means $X\gg Y$ and $Y\gg X$. 


\section{Preliminaries}

Throughout this paper, the letter $p$ always denotes an odd prime, and $r$ a positive integer. 

Let
\[
G_r:=SO_2(\bZ/p^r\bZ)=\left\{\begin{bmatrix}
a &-b\\
b & a
\end{bmatrix}\,(\mod p^r):\,\, a^2+b^2\equiv 1\,(\mod p^r)\right\}.
\]
For $\bfx=(x_1,x_2)\in (\bZ/p^r\bZ)^2$, the distance function is denoted by $\|\bfx\|:=x_1^2+x_2^2\, (\mod p^r)$. Then $\|g\bfx\|\equiv \|\bfx\|\, (\mod p^r)$ for any $g\in G_r$ and $\bfx\in (\bZ/p^r\bZ)^2$.  Write the orbit of $\bfx$ by
\[
\textsf{orb}_r(\bfx) = \big\{\theta \bfx:\, \theta \in G_r\big\},
\]
and the stabilizer of $\bfx$ by
\[
\textsf{stab}_r(\bfx) = \{\theta\in G_r:\, \theta \bfx \equiv \bfx\, (\mod p^r)\}.
\]
For any $j\in \bZ/p^r\bZ$, let $C_{j,r}$ be the circle centered at the origin of radius $j$, i.e.
\[
C_{j,r}=\{\bfx\in (\bZ/p^r\bZ)^2:\, \|\bfx\|\equiv j\,(\mod p^r)\}.
\]

For $x\in \bZ/p^r\bZ$ and $u\in \{0,1,\ldots,r-1\}$, the expression $v_p(x)=u$ means that $x\equiv 0\,(\mod p^u)$ and $x\not\equiv 0\,(\mod p^{u+1})$. We also denote $v_p(0) = r$ for $0\in \bZ/p^r\bZ$. When $\bfx=(x_1,x_2)$, we write $v_\bfx=\min\{v_p(x_1),\, v_p(x_2)\}$. Then the vector $\bfx$ can be expressed as $\bfx=p^{v_\bfx}\tilde{\bfx}$, where $\tilde{\bfx}$ is a vector in $(\bZ/p^{r-v_\bfz}\bZ)^n$ such that $v_{\tilde{\bfx}}=0$.

We recall the following facts from \cite{P.X}.
\begin{lemma}[\cite{P.X}, Lemma 4.5] \label{lem_staborb}
Let $p\equiv 3 \, (\mod 4)$. For any $\bfo\neq \bfx\in (\bZ/p^r\bZ)^2$, let $\bfx=p^{v_\bfx}\tilde{\bfx}$ for some $\tilde{\bfx}\in (\bZ/p^{r-v_\bfx}\bZ)^2$ with $v_{\tilde{\bfx}}=0$. Then $\textsf{orb}_r(\bfx)  = p^{v_\bfx}C_{\|\tilde{\bfx}\|,r-v_\bfx}$, 
and
\[
|\textsf{stab}_r(\bfx)| = p^{v_\bfx},\qquad |\textsf{orb}_r(\bfx)| =  p^{r-v_\bfx}\l(1+\frac{1}{p} \r).
\]
\end{lemma}
\begin{lemma}[\cite{P.X}, Lemma 4.9]  \label{lem_staborb1}
Let $p\equiv 1 \, (\mod 4)$.  For any $\bfo\neq \bfx\in (\bZ/p^r\bZ)^2$, let $\bfx=p^{v_\bfx}\tilde{\bfx}$ for some $\tilde{\bfx}\in (\bZ/p^{r-v_\bfx}\bZ)^2$ with $v_{\tilde{\bfx}}=0$. Then $\textsf{orb}_r(\bfx) =  p^{v_\bfx}\textsf{orb}_{r-v_\bfx}(\tilde{\bfx})$, and 
\[
|\textsf{stab}_r({\bfx})| = p^{v_\bfx},\qquad |\textsf{orb}_r(\bfx)| =  p^{r-v_\bfx}\l( 1-\frac{1}{p}\r).
\]
\end{lemma}

For a function $f:\, (\Zbb /p^r \Zbb)^2 \rightarrow \Cbb$, the Fourier transform $\widehat{f}:\,(\Zbb /p^r \Zbb)^2 \rightarrow \Cbb$ is defined by
\[
\widehat{f}(\bfm):=p^{-2r}\sum_{\bfx\in (\Zbb /p^r \Zbb)^2} \chi_r (-\bfm\cdot \bfx)f(\bfx), 
\]
where $\chi_r (x) = e^{\frac{2\pi ix}{p^r}}$ $(x\, \, \mod p^r)$.

We have the following basic properties of $\chi$ and $\widehat{f}$.
\begin{itemize}
    \item The orthogonality property:
\[
\sum_{{\bf \alpha } \in (\Zbb /p^r\Zbb)^2}\chi_r({\bf \beta }\cdot {\bf \alpha }) =
\begin{cases}
    p^{2r} ,& \text{if}\quad {\bf \beta } \equiv \bfo \, (\mod p^r) ,\\
    0 ,& \text{if}\quad {\bf \beta } \not \equiv \bfo \,(\mod p^r). 
\end{cases}
\]
\item The Fourier inversion formula:
    \[ f(\bfx)= \sum_{\bfm\in (\Zbb /p^r\Zbb)^2} \widehat{f}(\bfm)\chi_r(\bfm\cdot \bfx). \]
\item The Plancherel formula:
    \[ \sum_{\bfm\in (\Zbb /p^r \Zbb)^2} \lv \widehat{f} (\bfm)\rv^2 =p^{-2r} \sum_{\bfx\in (\Zbb /p^r \Zbb)^2} \lv f(\bfx) \rv^2. \]
\end{itemize}



For a set $A$, by abuse of notation, we also denote its characteristic function by $A(x)$, i.e. $A(x)=1$ if $x\in A$ and $A(x)=0$ if $x\not\in A$.

Now we state a more general restriction problem,.  

Let $\bfm\in (\bZ/p^r\bZ)^2$. When $r$ is given, we also use the simplified notation $V_\bfm$ for $\orb_r(\bfm)$. Let $d\sigma_{V_\mathbf{m}}$ be the corresponding surface measure on $V_\mathbf{m}$ and for all $f: V_\bfm \rightarrow \Cbb$, define
\[ (fd\sigma_{V_\bfm} )^\vee (\bfy) := \frac{1}{\lv V_\bfm\rv} \sum_{\bfx \in V_\bfm} f(\bfx ) \chi_r(\bfy \cdot \bfx) .\]
As computed in \cite{P.X}, we have

\begin{align*}
&\sum_{\bfm\,(\mod p^r)}|(fd\sigma_{V_\bfm})^\vee(\bfm)|^4=\frac{p^{2r}}{|V_{\bfm}|^4}\sum_{\xi,\xi',\eta,\eta'\in V_\bfm \atop \xi-\eta\equiv \xi'-\eta'\,(\mod p^r)}f(\xi)f(\xi')\overline{f(\eta)f(\eta')}.
\end{align*}
If $f$ is the characteristic function of a set $A\subset V_\bfm$, then $|V_{\bfm}|^4p^{-2r}\sum_{\bfm\,(\mod p^r)}|(fd\sigma_{V_\bfm})^\vee(\bfm)|^4$ counts the number of energy quadruples in $A$.

The following theorems 
give upper bounds of that sum for general functions $f$. 

\begin{theorem}[\cite{P.X}, Theorem 4.1]\label{restriction3}
Let $p\equiv 3\, (\mod 4)$ be a prime and  $r \ge 1$ be an integer. Let $\bfm\in (\bZ/p^r\bZ)^2$ be such that $\bfm\neq \bfo$. Then
         \[\left(\sum_{\bfx\in (\mathbb{Z}/p^r\mathbb{Z})^2}|(fd\sigma_{V_\mathbf{m}})^\vee(\bfx)|^4\right)^{\frac{1}{2}}\ll  p^{-\frac{r-3v_\bfm+1}{2}}\sum_{\bfx\in V_\mathbf{m}}|f(\bfx)|^2.\]
\end{theorem}

\begin{theorem}[\cite{P.X}, Theorem 4.2]\label{restriction6}
Let $p\equiv 1 \, (\mod 4)$ be a prime and $r \ge 1$ be an integer. Let $\bfm\in (\bZ/p^r\bZ)^2$ be such that $\bfm\neq \bfo$.  Suppose that $\bfm=p^{v_m}\tilde{\bfm}$ with $\tilde{\bfm}\in (\bZ/p^{r-v_m}\bZ)^2$. 

If $\|\tilde{\bfm}\|\not\equiv 0\,(\mod p)$, then
\[ \left( \sum_{\bfx\in (\bZ /p^r\bZ)^2} \vert (fd\sigma_{V_\bfm})^{\vee}(\bfx) \vert^4\right)^{\frac{1}{2}} \ll p^{-\frac{ r-3v_\bfm+1}{2}}\sum_{\bfx\in V_\bfm} \vert f(\bfx)\vert^2 .\]
    

If $\Vert \tilde{\bfm} \Vert \equiv 0 \, (\mod p)$,we have 
        \[  \left( \sum_{\bfx\in (\bZ /p^r\bZ)^2} \vert (fd\sigma_{V_\bfm})^{\vee}(\bfx) \vert^4\right)^{\frac{1}{2}} \ll p^{-\frac{r-3v_\bfm}{2}}\sum_{\bfx\in V_\bfm} \vert f(\bfx)\vert^2 .\] 
\end{theorem}




    
\begin{remark}\label{rm25}
  In comparison between Theorem \ref{restriction3} and Theorem \ref{restriction6}, the latter is weaker, which comes from elements $\mathbf{m}$ with $\Vert \tilde{\bfm} \Vert \equiv 0 \, (\mod p)$. 
\end{remark}

The next two lemmas are duality versions of Theorem \ref{restriction3} and Theorem \ref{restriction6}.

\begin{lemma} \label{lem_Ehat_Vm}
Let $p\equiv 3\, \pmod 4$. Let $E\subset (\Z/p^r\Z)^2$ and $\bfm\in (\Z/p^r\Z)^2$ be such that $\bfm\neq \bfo$. Then
\[
\sum_{\bfx \in V_\bfm} \lv \widehat{E}(\bfx)\rv^2 \ll p^{-\frac{5r+v_\bfm+1}{2}}\lv E\rv^{3/2}.
\]
\end{lemma}

\begin{proof}
By H{\" o}lder's inequality and  Theorem \ref{restriction3}, we have 
\begin{align*}
    &\frac{p^{2r}}{\lv V_\bfm\rv}\sum_{\bfx \in V_\bfm} \lv \widehat{E}(\bfx) \rv^2 =  \frac{p^{2r}}{\lv V_\bfm\rv}\sum_{\bfx \in V_\bfm} \widehat{E}(\bfx) \overline{\widehat{E}(\bfx)} = \sum_{\bfy \in (\Zbb/p^r \Zbb)^2} E(\bfy) \overline{\frac{1}{\lv V_\bfm\rv}\sum_{\bfx \in V_\bfm} \chi_r(\bfy \cdot \bfx) \widehat{E}(\bfx)} \\ &= \sum_{\bfy \in (\Zbb /p^r \Zbb)^2}E(\bfy) (\widehat{E}d\sigma_{V_\bfm})^\vee (\bfy) 
    \le \l( \sum_{\bfy \in (\Zbb/p^r \Zbb)^2} \lv E(\bfy)\rv^{\frac{4}{3}} \r)^{\frac{3}{4}} \l( \sum_{\bfy \in (\Zbb/p^r \Zbb)^2}\lv (\widehat{E}d\sigma_{V_\bfm})\rv^4 \r)^{\frac{1}{4}}\\
    & \ll |E|^{\frac{3}{4}} \l( p^{-\frac{r-3v_\bfm+1}{2}}\sum_{\bfx \in V_\bfm}\lv \widehat{E}(\bfx )\rv^2 \r)^{\frac{1}{2}}.
\end{align*}

Recall from Lemma \ref{lem_staborb} that $\lv V_\bfm\rv \sim p^{r-v_\bfm}$, the previous inequality implies
\begin{align*}
    p^{r+v_\bfm} \sum_{\bfx \in V_\bfm} \lv \widehat{E}(\bfx)\rv^2 \ll |E|^{\frac{3}{4}} \left(p^{\frac{-r+3v_\bfm -1}{2}} \sum_{\bfx \in V_\bfm} \lv \widehat{E}(\bfx)\rv^2\right)^{\frac{1}{2}}. 
\end{align*}
It follows that
\[  \sum_{\bfx \in V_\bfm} \lv \widehat{E}(\bfx)\rv^2 \ll p^{-\frac{5r+v_\bfm+1}{2}}\lv E\rv^{\frac{3}{2}}.\]
\end{proof}
\begin{lemma}\label{lem_Ehat_Vm1}
    Let $p\equiv 1\, (\mod 4)$ be a prime and $r\ge 1$ be an integer. Let $E\subset (\bZ/p^r\bZ)^2$ and $\mathbf{m}\in (\bZ/p^r\bZ)^2$ be such that $\bfm\neq \bfo$. Suppose that $\bfm = p^{v_\bfm}\tilde{\bfm}$, where $v_{\tilde{\bfm}} =0$.

If $\lV \tilde{\bfm}\rV \not\equiv 0 \pmod p$, then \[
\sum_{\bfx \in V_\bfm} \lv \widehat{E}(\bfx)\rv^2 \ll p^{-\frac{5r+v_{\bfm}+1}{2}}|E|^{\frac{3}{2}};
\]

If $\lV \tilde{\bfm}\rV \equiv 0 \pmod p$, then \[
\sum_{\bfx \in V_\bfm} \lv \widehat{E}(\bfx)\rv^2 \ll p^{-\frac{5r+v_{\bfm}}{2}}|E|^{\frac{3}{2}}.
\]
\end{lemma}
\begin{proof}
The proof is same as that of Lemma \ref{lem_Ehat_Vm}, with applying Theorem \ref{restriction6}.  
\end{proof}

\begin{corollary}\label{key-2-d}
Let $p$ be a prime, and 
    Let $p\equiv 3 \pmod 4$. Let $A, B$ be sets in $(\Zbb /p^r \Zbb)^2$.  Then we have 
    \[\sum_{\bfm \ne \bfo}\sum_{
    \bfm' \in V_\bfm}p^{v_\bfm} \lv \widehat{A}(\bfm)\rv^2 \lv \widehat{B}(\bfm ')\rv^2\ll p^{-4r-1}|A||B|^{\frac{3}{2}}.\]
\end{corollary}

\begin{proof} 
Applying Lemma \ref{lem_Ehat_Vm} and Plancherel formula, as above, one has
\begin{align*}
   & \sum_{\substack{\bfm \ne \bfo,\\
    \bfm' \in V_\bfm}}p^{v_\bfm}\lv \widehat{A}(\bfm)\rv ^2\lv \widehat{B}(\bfm')\rv^2 \le \sum_{\bfm}|\widehat{A}(\bfm)|^2 \cdot \max_{\bfz \ne \bfo} \l( p^{v_\bfz} \sum_{\bfm '\in V_\bfz} |\widehat{B}(\bfm')|^2 \r)  \\ 
    &\qquad\qquad \ll p^{-2r}|A| \cdot \max_{\bfz \ne \bfo} \l( p^{v_\bfz} \cdot p^{-\frac{5r+v_\bfz +1}{2}}|B|^{\frac{3}{2}} \r) \ll p^{-4r-1}|A||B|^{\frac{3}{2}},
\end{align*}
by noting that $\max v_\bfz = r-1$ for $\bfz \ne \bfo$. Thus, the theorem follows.
\end{proof}

\begin{corollary}\label{keyy-2-d}
    Let $p\equiv 1 \pmod 4$ be a prime, let $A, B$ 
    be sets in $(\Zbb /p^r \Zbb)^2$.  Then we have
    \[ \sum_{\bfm \ne \bfo} \sum_{\bfm' \in V_\bfm } p^{v_\bfm} \lv \widehat{A}(\bfm ) \rv \lv \widehat{B} (\bfm') \rv \ll p^{-4r-1/2}|A||B|^{3/2}.\]
\end{corollary}
\begin{proof}
The proof is same as that of Corollary \ref{key-2-d}, with applying Lemma \ref{lem_Ehat_Vm1}.  
\end{proof}

\section{Incidences between points and rigid-motions}

Recall that $G_r=SO_2(\Zbb /p^r \Zbb )$. We denote the set of rigid motion over $\Z/p^r\Z$ by 
\[
\mathcal{R}_r :=\{(g,\bfz):\, g\in G_r,\,\mathbf{z}\in (\Zbb /p^r \Zbb)^2\}. 
\]
Indeed, it is a semi-direct product of groups (see \cite{P.X.2024} for example). As a set, one has $\mathcal{R}_r = G_r\times (\Zbb /p^r \Zbb)^2$ and $|\mathcal{R}_r|\sim p^{3r}$.

We say a pair of points $(\bfx,\bfy)\in (\Zbb /p^r \Zbb )^2 \times (\Zbb /p^r \Zbb )^2$ is incident to a rigid-motion $(g, \mathbf{z})$ if $g\mathbf{y}+\mathbf{z}=\mathbf{x}$. For $P = A \times B \subset (\Zbb /p^r \Zbb )^2 \times (\Zbb /p^r \Zbb )^2$ and $R \subset \mathcal{R}_r $, we define
\[ \mathcal{I}(P,R) := \Big\{ \big((g,\bfz),(\bfx,\bfy)\big) \in R\times P:\, g\bfy +\bfz =\bfx  \Big\}\]
as the number of incidences between $P$ and $R$. This section is devoted to bounding $\mathcal{I}(P, R)$ from above and below. 

We first state a universal bound that will be proved by using a standard discrete Fourier analysis argument, followed by an improvement obtained via Corollary \ref{key-2-d}.

\begin{theorem}\label{thm:9.5}
    Let $p$ be an odd prime, let $P=A\times B \subset (\Zbb /p^r\Zbb)^2 \times (\Zbb /p^r \Zbb)^2$ and $ R \subset \mathcal{R}_r $. Then
   \[ \lv \mathcal{I}(P,R) - \frac{|P||R|}{p^{2r}} \rv \ll p^{\frac{3r-1}{2}}|P|^{\frac{1}{2}}|R|^{\frac{1}{2}}.\]
\end{theorem}

\begin{proof}
We have
\begin{align*}
    &\mathcal{I}(P, R)=\sum_{\substack{(\bfx, \bfy)\in P, \\ (g, \bfz)\in R}}1_{\bfx=g\bfy+\bfz}=\frac{1}{p^{2r}}\sum_{\bfm \in (\Zbb /p^r \Zbb)^2}\sum_{\substack{(\bfx, \bfy)\in P,\\ (g, \bfz)\in R}}\chi_r\left(\bfm\cdot (\bfx-g\bfy-\bfz )\right)\\
 &=\frac{|P||R|}{p^{2r}}+\frac{1}{p^{2r}}\sum_{\bfm \in (\Zbb /p^r \Zbb)^2\setminus \{ \bfo \}}\sum_{\substack{(\bfx, \bfy)\in P,\\ (g, \bfz)\in R}}\chi_r(\bfm \cdot(\bfx-g\bfy-\bfz))\\
 &=\frac{|P||R|}{p^{2r}}+p^{2r}\sum_{\bfm\ne \bfo }\sum_{(g, \bfz)\in R}\widehat{P}(-\bfm, g\bfm)\chi_r(-\bfm \bfz)=:I+II,
\end{align*}
where 
\[
\widehat{P}(\bfu, \bfv)=p^{-4r}\sum_{(\bfx, \bfy) \in P} \chi_r(-\bfx \cdot \bfu-\bfy \cdot \bfv).
\]
We next bound the second term. By the Cauchy-Schwarz inequality, we have 
\begin{align*}
  II&= p^{2r}\sum_{(g, \bfz)\in R}\sum_{\bfm\ne \bfo }\widehat{P}(-\bfm, g\bfm)\chi_r(-\bfm \bfz)\\
  &\le p^{2r} |R|^{1/2}\left(\sum_{(g, \bfz)\in \mathcal{R}_r}\sum_{\bfm_1, \bfm_2\ne \bfo} \widehat{P}(-\bfm_1, g\bfm_1)\overline{\widehat{P}(-\bfm_2, g\bfm_2)}\chi_r(\bfz \cdot (-\bfm_1+\bfm_2))\right)^{1/2}\\
  &= p^{2r} |R|^{1/2}\left(\sum_{g\in G_r}\sum_{\bfm_1, \bfm_2\ne \bfo} \widehat{P}(-\bfm_1, g\bfm_1)\overline{\widehat{P}(-\bfm_2, g\bfm_2)}\sum\limits_{\bfz\in (\bZ/p^r\bZ)^2}\chi_r(\bfz \cdot (-\bfm_1+\bfm_2))\right)^{1/2}\\
  &=p^{2r} |R|^{1/2}\left(p^{2r}\sum_{g\in G_r}\sum_{\bfm \ne \bfo}|\widehat{P}(\bfm, -g\bfm)|^2\right)^{1/2}.
\end{align*}

Since $P=A\times B$, we have
\[\widehat{P}(\bfm, -g\bfm)=\widehat{A}(\bfm)\widehat{B}(-g\bfm).\]
Thus, following Lemma \ref{lem_staborb1} and Plancherel fomula, we have
\begin{align*}
    &\sum_{g\in G_r}\sum_{\bfm \ne \bfo}|\widehat{P}(\bfm, -g\bfm)|^2 =\sum_{\bfm \ne \bfo}|\widehat{A}(\bfm)|^2\sum_{g\in G_r}|\widehat{B}(-g\bfm)|^2\\ 
    & \ll \sum_{\bfm \ne \bfo}|\widehat{A} (\bfm)|^2 \sum_{\bfm' \in V_\bfm}p^{v_{\bfm}} |\widehat{B}(\bfm')|^2 \le \sum_{\bfm \ne \bfo}p^{v_{\bfm}}|\widehat{A} (\bfm)|^2 \sum_{\bfm'} |\widehat{B}(\bfm')|^2 \\
    &\le p^{r-1}\sum_{\bfm} |\widehat{A}(\bfm)|^2 \sum_{\bfm'} |\widehat{B}(\bfm')|^2 = p^{r-1}\cdot \frac{|A|}{p^{2r}}\cdot \frac{|B|}{p^{2r}} =  \frac{|A||B|}{p^{3r+1}}.
\end{align*}
Here we have used the fact that the stabilizer of a non-zero element $\bfm$ in $SO_2(\Zbb /p^r \Zbb)$ is $\sim p^{v_{\bfm}}\le p^{r-1}$.

Finally, we obtain that
\begin{align*}
    II &  \ll p^{3r}|R|^{\frac{1}{2}}\l( \frac{|A||B|}{p^{3r+1}} \r)^{\frac{1}{2}} = p^{\frac{3r-1}{2}}|R|^{\frac{1}{2}}|P|^{\frac{1}{2}}.
\end{align*} 
This completes the proof.
\end{proof}

Next, we estimate the incidences in another way. 

\begin{theorem}\label{thm:9.4}
   Let $p$ be an odd prime. Let $P=A\times B \subset (\Zbb /p^r\Zbb)^2 \times (\Zbb /p^r \Zbb)^2$ and $ R \subset \mathcal{R}_r$. 

If $p\equiv 3\pmod 4$, then  \[ \lv \mathcal{I}(P,R) - \frac{|P||R|}{p^{2r}} \rv \ll p^{r-\frac{1}{2}}|P|^{\frac{1}{2}}|R|^{\frac{1}{2}}|B|^{\frac{1}{4}} .\]

If $p\equiv 1\pmod 4$, then  \[ \lv \mathcal{I}(P,R) - \frac{|P||R|}{p^{2r}} \rv \ll p^{r-\frac{1}{4}}|P|^{\frac{1}{2}}|R|^{\frac{1}{2}}|B|^{\frac{1}{4}} .\]  
\end{theorem}

\begin{proof}
As previous, we have 
\begin{align*}
    &\mathcal{I}(P, R)=
\frac{|P||R|}{p^{2r}}+p^{2r}\sum_{\bfm\ne \bfo }\sum_{(g, \bfz)\in R}\widehat{P}(-\bfm, g\bfm)\chi_r(-\bfm \cdot \bfz)=:I+II,
\end{align*}
and 
\begin{align*}
  II&
 \leq p^{3r}|R|^{\frac{1}{2}}\left(\sum_{g\in G_r}\sum_{\bfm \ne \bfo}|\widehat{P}(\bfm, -g\bfm)|^2\right)^{\frac{1}{2}}\\
 &\ll p^{3r}|R|^{\frac{1}{2}}\left(\sum_{\substack{\bfm \ne \bfo,\\
    \bfm' \in V_\bfm}} p^{v_{\bfm}} |\widehat{A}(\bfm)|^2|\widehat{B}(\bfm')|^2\right)^{\frac{1}{2}}.
\end{align*}
%

To prove the first assertion, we use Corollary \ref{key-2-d} and obtain
\begin{align*}
    II &  \ll p^{3r}|R|^{\frac{1}{2}}\l( p^{-4r-1}|A||B|^{\frac{3}{2}}\r)^{\frac{1}{2}}= p^{2r-\frac{1}{2}}|R|^{\frac{1}{2}}|P|^{\frac{1}{2}} |B|^{\frac{1}{4}}.
\end{align*} 

To prove the second assertion, we use Corollary \ref{keyy-2-d} and obtain
\begin{align*}
    II &  \ll p^{3r}|R|^{\frac{1}{2}}\l( p^{-4r-\frac{1}{2}}|A||B|^{\frac{3}{2}} \r)^{\frac{1}{2}}= p^{r-\frac{1}{4}}|R|^{\frac{1}{2}}|P|^{\frac{1}{2}} |B|^{\frac{1}{4}}.
\end{align*} 
This completes the proof.
\end{proof}

\begin{remark}
Compared to Theorem \ref{thm:9.5}, Theorem \ref{thm:9.4} gives improvements in the ranges $|B|<p^{2r}$ and $|B|<p^{2r-1}$, corresponding to $p\equiv 3\pmod 4$ and $p\equiv 1\pmod 4$, respectively. 
\end{remark}

\section{Proof of main theorems}

\begin{proof} [Proof of Theorem \ref{thm:9.9}]
Since $\delta_A^{1/2}\cdot\delta_B\ge 2p^{-1}$, we have $\delta_A\ge 4p^{-2}$. Thus, $|A|\ge 4p^{2r-2}$. The number of elements $\bfx$ in $(\Z/p^r\bZ)^2$ with $v_\bfx\neq 0$ does not exceed $p^{2r-2}$. So, without loss of generality, we may assume that $|A|\geq 3p^{2r-2}$ and  $v_\bfx =0$ for any $\bfx \in A$. 

    Let 
    \[ \widetilde{G} = \lo  g \in G_r\colon \lv \lo \bfz \in (\Zbb /p^r \Zbb)^2 \colon  B \cap (g A +\bfz ) =\emptyset \ro \rv \ge \frac{p^{2r}}{2} \ro .\]
    It is sufficient to show that $|\widetilde{G}| \ll |G_r|$. 

Let $\widetilde{R}$ be the set of pairs $(g ,\bfz ) \in \mathcal{R}_r$ such that $g \in \widetilde{G}$ and  $ B \cap (g A +\bfz ) = \emptyset$. It is clear that $\mathcal{I}(B \times A ,\widetilde{R}) = 0$. Alternatively, from Theorem \ref{thm:9.4} (1), we have
    \begin{equation} \label{thmproof_eq1}
    \lv \mathcal{I}(B\times A ,\widetilde{R}) -\frac{|A||B||\widetilde{R}|}{p^{2r}}\rv  \ll p^{r-\frac
    {1}{2}}|A|^{\frac{3}{4}}|B|^{\frac{1}{2}}|\widetilde{R}|^{\frac{1}{2} }.
    \end{equation}
    Thus, we deduces that
    \[|\widetilde{R}| \ll p^{6r-1} |A|^{-\frac{1}{2}}|B|^{-1}. \]

On the other hand, one sees by the construction of $\widetilde{G}$ that $| \widetilde{R}| \ge |\widetilde{G}| \cdot p^{2r}/2$, we have
    \begin{equation} \label{thmproof_eq2}
    |\widetilde{G} | \ll p^{4r-1} |A|^{-\frac{1}{2}}|B|^{-1}.
    \end{equation}
It leads to $|\widetilde{G} |\ll p^r$, in view of the condition $|A|^{1/2}|B| \ge 2p^{3r-1}.$
    
    
    
\end{proof}

In above arguments, if we use Theorem \ref{thm:9.4} (2), then the condition $|A|^{1/2}|B|\ge 2p^{3r-\frac{1}{2}}$ is required. By a direct computation, we can see that this condition is worse than those of Theorem \ref{thm:9.10}.  

\begin{proof} [Proof of Theorem \ref{thm:9.10}]
The proof is same as that of Theorem  \ref{lem_Ehat_Vm}. In this case, one deduces from $\delta_A\cdot\delta_B\ge 2p^{-1}$ that $\delta_A\ge 2p^{-1}$. Then the cardinality of $A$ is at least $2p^{2r-1}$, which is much larger than the number of elements $\bfx$ with $v_\bfx\neq 0$. From Theorem \ref{thm:9.5}, the bound on the right-hand side of \eqref{thmproof_eq1} is replaced by $p^{\frac{3r-1}{2}}|A|^{\frac{1}{2}}|B|^{\frac{1}{2}}|\widetilde{R}|^{\frac{1}{2}}$. And \eqref{thmproof_eq2} becomes 
\[
|\widetilde{G}| \ll p^{-2r}|\widetilde{R}| \ll p^{5r-1}|A|^{-1}|B|^{-1},
\]
which gives $|\widetilde{G}| \ll p^r$ provided that $|A||B|\geq 2p^{4r-1}$.
\end{proof}

\section{Acknowledgements} Thang Pham would like to thank the Vietnam Institute for Advanced Study in Mathematics (VIASM) for the hospitality and for the excellent working condition.

Thang Pham, Nguyen Duy Phuong, and Le Quang Ham were supported by Vietnam National Foundation for Science and Technology Development (NAFOSTED) under grant number 101.99--2021.09.

\section*{Appendix: Proof of Theorem \ref{restriction6}}
As shown in \cite{P.X}, the proof of Theorem \ref{restriction6} is almost identical with that of Theorem \ref{restriction3} for the case $p\equiv 3\pmod 4$. For clarity, we provide a detailed proof here. 


In this section, the letter $p$ is always a prime with $p \equiv 1 (\mod 4)$, and $r$ is a positive integer. Recall that $G_r=SO_2(\Z/p^r\Z)$.

\begin{lemma}[\cite{P.X}, Lemma 2.1] \label{lem_Hensel}
For any prime $p$, let
\[
\mathbf{G}(\bfx) = \big(\widetilde{G}(x_1,\ldots,x_n),\ldots,G_m(x_1,\ldots,x_n)\big)
\]
be a map from $\bZ^n$ to $\bZ^m$, with $G_i$ polynomials with integer coefficients. Let $l$ be a positive integer and $\bfy\in \bZ^n$. Suppose that $\mathbf{G}(\bfy)\equiv \bfo$ $(\mod p^l)$. Let $R=\text{rank} J_{\mathbf{G}}(\bfy)$, where $J_{\mathbf{G}}(\bfy)$ is the Jocobi matrix modulo $p$, i.e.,
\[
J_{\mathbf{G}}(\bfy)=\begin{bmatrix}
\frac{\partial \widetilde{G}}{\partial x_1}(\bfy) &\frac{\partial \widetilde{G}}{\partial x_2}(\bfy) &\ldots &\frac{\partial \widetilde{G}}{\partial x_n}(\bfy) \\
\frac{\partial G_2}{\partial x_1}(\bfy) &\frac{\partial G_2}{\partial x_2}(\bfy) &\ldots &\frac{\partial G_2}{\partial x_n}(\bfy) \\
\ldots &\ldots &\ldots &\ldots \\
\frac{\partial G_m}{\partial x_1}(\bfy) &\frac{\partial G_m}{\partial x_2}(\bfy) &\ldots &\frac{\partial G_m}{\partial x_n}(\bfy)
\end{bmatrix}\qquad (\mod p).
\]
Then
\begin{equation} \label{eq0}
\#\big\{\bfz\,(\mod p^k):\, \mathbf{G}(\bfy+p^l\bfz)\equiv \mathbf{0}\,(\mod p^{l+k})\big\} \leq p^{k(n-R)}
\end{equation}
for any integer $k\geq 1$. When $R=m$, the ``$\leq$'' can be replaced by ``$=$''.
\end{lemma}

\begin{lemma}\label{lem_car_Gr1}
    We have $\left\vert G_r \right\vert =p^r (1-1/p)$.
\end{lemma}

\begin{proof}
   For $r=1$, we have $|G_1|=p-1$. Now consider the circumstances that $r\geq 2$. For any $\theta=\begin{bmatrix}a & -b\\b &a\end{bmatrix} \in G_r$, there is  some $\theta_0 = \begin{bmatrix}a_0 & -b_0\\b_0 &a_0\end{bmatrix}\in G_1$ such that $\theta\equiv \theta_0\,(\mod p)$. Applying Lemma \ref{lem_Hensel} to $(a_0,b_0)$ and the polynomial $F(x,y)=x^2+y^2-1$, one obtains that $(\nabla F)(a_0,b_0) = (2a_0, 2b_0) \not\equiv (0,0) \,(\mod p)$, since $a_0^2+b_0^2\equiv 1 \pmod p$. Thus,
\[
\#\left\{(z_1,z_2)\,(\mod p^{r-1}):\, (a_0+pz_1)^2+(b_0+pz_2)^2 \equiv 1\,(\mod p^r)\right\} = p^{r-1}.
\]
It follows that
\[
|G_r|=p^{r-1}|G_1|=p^{r}(1-1/p). 
\]
\end{proof}


\begin{proof}[Proof of Lemma \ref{lem_staborb1}]
    Here $0 \le v_\bfx \le r-1$. The equation $\theta \bfx \equiv \bfx (\mod p^r)$ is equivalent to 
    \begin{equation}
        \label{eq_rot3}
        \begin{bmatrix}
            a-1 & -b \\
            b & a-1
        \end{bmatrix} \begin{bmatrix}
            p^{v_\bfx} \tilde{x_1} \\
            p^{v_\bfx} \tilde{x_2}
        \end{bmatrix} \equiv \begin{bmatrix}
            0 \\
            0
        \end{bmatrix} \quad (\mod p^r) ,
    \end{equation}
    or equivalently,
    \begin{equation} \label{eq_rot4}
\begin{bmatrix}
\tilde{x_1} &-\tilde{x_2}\\
\tilde{x_2} & \tilde{x_1}
\end{bmatrix}
\begin{bmatrix}
a-1\\
b
\end{bmatrix}\equiv
\begin{bmatrix}
0\\
0
\end{bmatrix}\quad (\mod p^{r-v_\bfx}).
\end{equation}

If $\tilde{x_1}^2+\tilde{x_2}^2 \not\equiv 0 \, (\mod p) $, then the coefficient matrix is invertible modulo $p^r$, and
\[
\begin{bmatrix}
a\\
b
\end{bmatrix}\equiv
\begin{bmatrix}
1\\
0
\end{bmatrix},\quad (\mod p^{r-v_\bfx}).
\]
By Lemma \ref{lem_Hensel}, the number of $(a,b)\in (\bZ/p^r \bZ)^2$ satisfying $a^2+b^2\equiv 1 \pmod {p^r}$ and the above equivalence is exactly $p^{v_\bfx}$. 

If 
$\tilde{x_1}^2+\tilde{x_2}^2 \equiv 0 \, (\mod p)$, then one deduce from $(\tilde{x_1},\tilde{x_2}) \not\equiv (0,0) \pmod p$ that $\tilde{x_1}, \tilde{x_2}\not\equiv 0 \pmod p$. Therefore, let $u=a-1, v=b$, we have a system of
linear equation
\[ \begin{cases}
    \tilde{x_1}u -\tilde{x_2}v \equiv 0 \\
    \tilde{x_2}u +\tilde{x_1}v \equiv 0 \\
    (u+1)^2 +v^2 \equiv 1
\end{cases} \quad (\mod p^{r-v_\bfx}). \]
It follows that
\[\begin{cases}
    u \equiv \tilde{x_1}^{-1}\tilde{x_2}v \\
    2\tilde{x_1}\tilde{x_2}v \equiv 0
\end{cases} \quad (\mod p^{r-v_\bfx}).\]
Hence, $(u,v)\equiv (0,0) \pmod {p^{r-v_\bfx}}$, or equivalently, $(a,b)\equiv (1,0) \pmod {p^{r-v_\bfx}}$. By Lemma \ref{lem_Hensel} again, the number of such $(a,b)$ modulo $p^r$ is exactly $p^{v_\bfx}$.
So
\[
\big|\textsf{stab}_\bfx\big| = p^{v_\bfx}.
\]
It then follows that $|\textsf{orb}(\bfx)| = |G_r|/|\textsf{stab}_\bfx| = p^{r-v_\bfx}(1-1/p)$.

Moreover, for any $\theta\in G_r$, there is some $\theta_0\in G_{r-v_\bfx}$ such that $\theta \equiv \theta_0\,(\mod p^{r-v_\bfx})$. It can be verified that $\theta \tilde{\bfx} \equiv \theta_0 \tilde{\bfx}\, (\mod p^{r-v_\bfx})$. So 
\[
\textsf{orb}_r(\bfx) = \big\{\theta (p^{v_\bfx} \tilde{\bfx}):\, \theta \in G_r\big\} = \big\{p^{v_\bfx} \theta_0\tilde{\bfx}:\, \theta_0 \in G_{r-v_\bfx}\big\}.
\]
Note that $|G_{r-v_\bfx}|=p^{r-v_\bfx}(1-1/p)$ by Lemma \ref{lem_car_Gr1}, the elements on the right-hand side of the above formula give different members of the orbit. The proof is completed.
\end{proof}

\begin{lemma}\label{energy4}
    Let $p\equiv 1\, (\mod 4)$. Suppose that $ \bfm =p^{v_{\bfm}}\tilde{\bfm} \in (\bZ /p^r \bZ)^2$ with $\tilde{\bfm} \in \l( (\bZ /p^{r-v_{\bfm}} \bZ)^\ast \r)^2$. Then for any $\bfz \in (\bZ /p^r\bZ)^2$, we have
    \begin{equation} \label{eq_app_proof_1}
    \# \{ (\bfx ,\bfy )\in (\textsf{orb}_r (\bfm))^2 \colon \bfx -\bfy \equiv \bfz (\mod p^r) \} \ll \begin{cases}
        p^{r-v_{\bfm}-1}, & \text{if $\Vert \tilde{\bfm}\Vert \not\equiv 0\, (\mod p)$},\\
        p^{r-v_{\bfm}},  & \text{if $\Vert \tilde{\bfm} \Vert \equiv 0 \, (\mod p)$}.
    \end{cases} 
    \end{equation}
\end{lemma}
\begin{proof}
When $\bfm=\bfo$, the conclusion holds automatically. In the following, we assume that $\bfm\neq \bfo$. By Lemma \ref{lem_staborb1}, one has $\textsf{orb}_r(\bfm)=p^{v_\bfm}\orb_{r-v_\bfm}(\tilde{m})$. Let us write $\bfx=p^{v_\bfm}\tilde{\bfx}$ and $\bfy=p^{v_\bfm}\tilde{\bfy}$, with $\tilde{\bfx},\,\tilde{\bfy} \in \orb_{r-v_\bfm}(\tilde{\bfm})$. 
When $v_{\bfz}<v_{\bfm}$, the equation $\bfx-\bfy\equiv \bfz \pmod {p^r}$ has no solution. When $v_{\bfz}\geq v_{\bfm}$, we denote $\bfz=p^{v_\bfm}\tilde{\bfz}$ and obtain that $\tilde{\bfx}-\tilde{\bfy}\equiv \tilde{\bfz} \pmod {p^{r-v_\bfm}}$. Noting that $v_{\tilde{\bfm}}=0$, it is sufficient to show that
\[ \# \{ (\tilde{\bfx} ,\tilde{\bfy} )\in (\textsf{orb}_r (\tilde{\bfm}))^2 \colon \tilde{\bfx} -\tilde{\bfy} \equiv \tilde{\bfz} (\mod p^{r-v_\bfm}) \} \ll \begin{cases}
        p^{(r-v_{\bfm})-v_{\tilde{\bfm}}-1}, & \text{if $\Vert \tilde{\bfm}\Vert \not\equiv 0\, (\mod p)$},\\
        p^{(r-v_{\bfm})-v_{\tilde{v_m}}},  & \text{if $\Vert \tilde{\bfm} \Vert \equiv 0 \, (\mod p)$}.
    \end{cases} \]
    Therefore, we may assume at the beginning of the proof that $v_\bfm=0$.

Since $\bfy$ is determined by $\bfx$ modulo $p^r$, the cardinality on the left-hand-side of \eqref{eq_app_proof_1} does not exceed $|\textsf{orb}_r (\tilde{\bfm})|$, which is $\sim p^r$ by Lemma \ref{lem_staborb1}. The second upper bound follows. 

Next, we consider the case $\|\bfm\|\not\equiv 0 \pmod p$. We have assumed that $v_\bfm=0$, so $|\textsf{stab}_r(\bfm)|=1$ by Lemma \ref{lem_staborb1}. For $\bfx ,\bfy \in \textsf{orb}_r(\bfm)$, there exist unique $a,b,a',b'\in \bZ /p^r \bZ$  such that $a^2+b^2 \equiv a'^2 +b'^2 \equiv 1\, (\mod p^r)$ and 
    \[ \begin{bmatrix}
        a & -b \\
        b & a
    \end{bmatrix}\bfm \equiv \bfx,\quad \begin{bmatrix}
        a' & -b'\\
        b' & a'
    \end{bmatrix}\bfm \equiv \bfy, \quad \pmod {p^r}.\]
    The set on the left-hand side of \eqref{eq_app_proof_1} involves a system of congruences
    \begin{equation}
        \label{eq_sys4}
        \begin{cases}
        (a-a')m_1 -(b-b')m_2 \equiv z_1,\\
        (b-b')m_1 + (a-a')m_2 \equiv z_2,\\
        a^2 + b^2 \equiv 1,\\
        a'^2 +b'^2 \equiv 1.
    \end{cases} 
    \end{equation} 
    We first consider \eqref{eq_sys4} modulo $p$. 

    Since $m_1^2+m_2^2\not\equiv 0 \pmod p$, either $m_2$ or $m_1 \not\equiv 0\, \pmod p$. Without loss of generality, we assume that $m_1 \not\equiv 0\, (\mod p)$, so one sees that
    \[ \begin{cases}
        a-a' \equiv m_1^{-1}(z_1 +(b-b')m_2),\\
        (b-b')m_1 + m_1^{-1}(z_1+(b-b')m_2)m_2 \equiv z_2, 
    \end{cases} \quad (\mod p), \]
    or equivalently,
    \[ \begin{cases}
        a-a' \equiv m_1^{-1}(z_1 +(b-b')m_2), \\
        (b-b')(m_1^2+m_2^2) \equiv z_2m_1-z_1m_2. 
    \end{cases} \quad (\mod p). \]
   Now $(a,b)$ is determined by $(a',b')$ and system \eqref{eq_sys4} has at most $2$ solutions modulo $p$.

    In the following, we will apply Hensel's lemma. The Jacobi matrix of $G(a,b,a',b') = ((a-a')m_1-(b-b')m_2-z_1 ,(b-b')m_1+(a-a')m_2-z_2,a^2+b^2-1,a'^2+b'^2-1)$ in $(a,b,a',b')$ is given by
    \[ \begin{bmatrix}
        m_1 & -m_2 & -m_1 & m_2\\
        m_2 & m_1 & -m_2 & -m_1\\
        2a & 2b & 0 & 0\\
0 & 0 & 2a' & 2b'
    \end{bmatrix}.\]
    By elementary operations, we obtain
    \[ \begin{bmatrix}
        0 & 0 & -m_1 & m_2\\
        0 & 0 & -m_2 & -m_1\\
        2a & 2b & 0 & 0\\
        2a' & 2b' & 2a' & 2b'
    \end{bmatrix}. \]
    In view of the fact that $m_1^2+m_2^2 \not\equiv 0 \, (\mod p)$, the rank of the above Jacobi matrix, modulo $p$, is at least $3$. By Lemma \ref{lem_Hensel}, the number of solutions of $(a,b,a',b')$ to \eqref{eq_sys4} modulo $p^r$ is at most $p^{r-1}$. The first upper bound then follows. 

\end{proof}

Next, we complete the proof of Theorem \ref{restriction6}.

\begin{proof}[Proof of Theorem \ref{restriction6}]
We have
    \begin{align*}
    &\sum_{\bfm\in (\mathbb{Z} / p^r\Zbb)^2}|(fd\sigma_{V_\bfm})^\vee(\bfm)|^4=\sum_{\bfm}\left\vert \frac{1}{|V_{\bfm}|}\sum_{\bfx\in V_{\bfm}}\chi_r(\bfm\cdot \bfx)f(\bfx)\right\vert^4\\
    &=\frac{p^{2r}}{|V_{\bfm}|^4}\sum_{\xi,\xi',\eta,\eta'\in V_{\bfm}\colon \xi-\eta=\xi'-\eta'}f(\xi)f(\xi')\overline{f(\eta)f(\eta')}
\end{align*}

Moreover,
\[
\sum_{\xi,\xi',\eta,\eta'\in V\colon \xi-\eta=\xi'-\eta'}f(\xi)f(\xi')\overline{f(\eta)f(\eta')} = \sum_{\zeta} \left|\sum_{\xi-\eta=\zeta}f(\xi)\overline{f(\eta)}V_{\bfm}(\xi)V_{\bfm}(\eta) \right|^2.
\]


For $\zeta\equiv \mathbf{0}\,(\mod p^r)$, we have
\[\sum_{\zeta\equiv \mathbf{0}\,(\mod p^r)} \left| \sum_{\xi-\eta=\zeta}f(\xi)\overline{f(\eta)}V_{\bfm}(\xi)V_{\bfm}(\eta)\right|^2\ll \left(\sum_{\xi\in V_{\bfm}}|f(\xi)|^2\right)^2.\]

For $\zeta\not\equiv \mathbf{0}\,(\mod p^r)$, the Cauchy-Schwarz inequality implies
\begin{align*}
\phantom{-}& \sum_{\zeta\not\equiv \mathbf{0}\,(\mod p^r)} \left| \sum_{\xi-\eta=\zeta}f(\xi)\overline{f(\eta)}V_{\bfm}(\xi)V_{\bfm}(\eta)\right|^2 \\
&\leq \sum_{\zeta\not\equiv \mathbf{0}\,(\mod p^r)} \left( \sum_{\xi-\eta=\zeta}V_{\bfm}(\xi)V_{\bfm}(\eta) \right) \sum_{\xi-\eta=\zeta}|f(\xi)|^2|f(\eta)|^2V_{\bfm}(\xi)V_{\bfm}(\eta).\\
\end{align*}
Now, to bound the sum $ \sum_{\xi-\eta=\zeta}V_{\bfm}(\xi)V_{\bfm}(\eta) $, we use Lemma \ref{energy4} and get
\begin{align*}
    & \sum_{\zeta\not\equiv \mathbf{0}\,(\mod p^r)} \left| \sum_{\xi-\eta=\zeta}f(\xi)\overline{f(\eta)}V_{\bfm}(\xi)V_{\bfm}(\eta)\right|^2 \\
    \ll & \begin{cases}
        p^{r-v_\bfm -1} \l( \sum_{\xi \in V_\bfm} |f(\xi)|^2 \r)^2 & \text{if $\lV \tilde{\bfm} \rV \not\equiv 0 \pmod p $}, \\
        p^{r-v_\bfm } \l( \sum_{\xi \in V_\bfm} |f(\xi)|^2 \r)^2 & \text{if $\lV \tilde{\bfm} \rV \equiv 0 \pmod p $}.
    \end{cases}
\end{align*}

So, one has 
\begin{align*}
    \left( \sum_{x\in (\bZ /p^r\bZ)^2} \vert (fd\sigma_{V_\bfm})^{\vee}(x) \vert^4\right)^{1/2} \ll  \frac{p^{r}}{\lv V_\bfm\rv^2 }  \begin{cases}
        p^{\frac{r-v_\bfm -1}{2}} \sum_{\xi \in V_\bfm } |f(\xi)|^2  & \text{if $\lV \tilde{\bfm} \rV \not\equiv 0 \pmod p $}, \\
        p^{\frac{r-v_\bfm}{2} }  \sum_{\xi \in V_\bfm} |f(\xi)|^2  & \text{if $\lV \tilde{\bfm} \rV \equiv 0 \pmod p $}.
    \end{cases}
\end{align*}

Next, Lemma \ref{lem_staborb1} gives that $|V_\bfm | \sim p^{r-v_\bfm}$. Hence, the theorem follows.


\end{proof}


\begin{thebibliography}{}

 \bibitem{CEHIK10} J. Chapman, M. B. Erdogan, D. Hart, A. Iosevich, and D. Koh, {\it Pinned distance sets, k-simplices,Wolff's exponent in finite fields and sum-product estimates}, Mathematische Zeitschrift, \textbf{271}(1--2) (2012), 63--93.
 
\bibitem{alex-fal}
L. Guth, A. Iosevich, Y. Ou, and H. Wang, \textit{On Falconer's distance set problem in the plane}, Inventiones mathematicae, \textbf{219}(3) (2019), 779--830.

\bibitem{KS}
D. Koh and H. Sun, \textit{Distance sets of two subsets of vector spaces over finite fields}, Proceedings of the American Mathematical Society, \textbf{143}(4) (2015), 1679--1692.

\bibitem{BL}
B. Lichtin, \textit{Distance and sum–product problems over finite p-adic rings}, Proceedings of the London Mathematical Society, \textbf{118}(6) (2019), 1450--1470.
\bibitem{BL2}
B. Lichtin, \textit{Averages of point configuration problems over finite p-adic rings}, Proceedings of the American Mathematical Society, \textbf{149}(7) (2021), 2825-2839.
\bibitem{BL3}
B. Lichtin, \textit{k-chain configurations of points over p-adic rings}, Proceedings of the American Mathematical Society, \textbf{151}(10)  (2023), 4113--4125.

\bibitem{mu2}
B. Murphy, G. Petridis, T. Pham, M. Rudnev, and S. Stevens, \textit{On the pinned distances problem over finite fields},  Journal of the London Mathematical Society, \textbf{105}(1) (2022), 469--499.

\bibitem{Mattila}
P. Mattila, \textit{Hausdorff dimension and projections related to intersections}, Publicacions
matematiques, \textbf{66}(1) (2022), 305--323.

\bibitem{M2023} P. Mattila, \textit{A survey on the Hausdorff dimension of intersections},  Mathematical and Computational Applications, \textbf{28}(2) (2023), 49.

 




\bibitem{P.X.2024} T. Pham, and B. Xue, {\it New-type Quasirandom Groups and Applications}, Journal of Fourier Analysis and Applications, 64 (2024).

\bibitem{P.X} T. Pham, and B. Xue, {\it On the distance problem over finite p-adic rings}, \href{https://arxiv.org/abs/2405.07325}{arXiv:2405.07325} (2024). 

\bibitem{P.Y.2023} T. Pham, and S. Yoo, {\it Intersection patterns and incidence theorems}, \href{https://arxiv.org/abs/2304.08004}{arXiv:2304.08004} (2023).


\end{thebibliography}
\end{document}